\documentclass[11pt]{article}
\usepackage{amssymb}
\usepackage{amsmath}
\usepackage{amsfonts}

\usepackage[body=16cm,top=4cm,bottom=4.5cm]{geometry}
\usepackage[english]{babel}

\newtheorem{theorem}{Theorem}

\newtheorem{definition}[theorem]{Definition}

\newtheorem{lemma}[theorem]{Lemma}

\newtheorem{remark}[theorem]{Remark}

\newenvironment{proof}[1][Proof]{\noindent\textbf{#1.} }{\ \rule{0.5em}{0.5em}}
\input{tcilatex}

\begin{document}

\begin{center}
\bigskip {\LARGE Regularity Properties of the Stochastic Flow of a Skew
Fractional Brownian Motion}

{ Oussama Amine\footnote{{ email: oussamaa@math.uio.no}}}$^{%
\text{,4}}${ , David R. Ba\~{n}os\footnote{{ email:
davidru@math.uio.no}}}$^{\text{,4}}${ \ and Frank Proske\footnote{%
{ email: proske@math.uio.no}}}$^{\text{,}}${ \footnote{{ %
Department of Mathematics, University of Oslo, Moltke Moes vei 35, P.O. Box
1053 Blindern, 0316 Oslo, Norway.}}}

\bigskip

{\large Abstract}
\end{center}

In this paper we prove, for small Hurst parameters, the higher order
differentiability of a stochastic flow associated with a stochastic
differential equation driven by an additive multi-dimensional fractional Brownian
noise, where the bounded variation part is given by the local time of the
unknown solution process. The proof of this result relies on Fourier analysis based variational
calculus techniques and on intrinsic properties of
the fractional Brownian motion.

\emph{keywords}: SDEs, Compactness criterion, generalized drift, Malliavin
calculus, reflected SDE's, stochastic flows.


\section{Introduction}

\bigskip Consider a $d-$dimensional fractional Brownian motion (fBm) 

\begin{equation*}
B_{t}^{H}=(B_{t}^{1,H},\ldots ,B_{t}^{d,H}),\text{ }0\leq t\leq T
\end{equation*}%

with Hurst parameter $H\in (0,1)$ constructed on some complete probability
space $(\Omega ,\mathcal{F},\mu ).$ Here $B_{\cdot }^{1,H},\ldots ,B_{\cdot
}^{d,H}$ are independent $1-$dimensional fractional Brownian motions, that
is centered Gaussian processes with a covariance structure given by%

\begin{equation*}
R_{H}(t,s)=E[B_{t}^{i,H}B_{s}^{i,H}]=\frac{1}{2}(s^{2H}+t^{2H}+\left\vert
t-s\right\vert ^{2H}).
\end{equation*}%

We mention that for $H=\frac{1}{2}$ the fBm is a standard Wiener process. If 
$H\neq \frac{1}{2}$, the fBm is neither a semimartingale nor a Markov
process. See e.g. \cite{Nualart} for more information about fBm.

In this paper we want to study the regularity of solutions $X_{t}^{x},$ $%
0\leq t\leq T$ to the stochastic differential equation (SDE)%

\begin{equation}
X_{t}^{x}=x+\alpha L_{t}(X^{x})\cdot \mathbf{1}_{d}+B_{t}^{H},\text{ }0\leq
t\leq T  \label{SkewfBm}
\end{equation}%

with respect to their initial condition $x\in \mathbb{R}^{d}$. Here the
Hurst parameter $H$ of the fBm is small, that is $H\in (0,1/2)$ , $\alpha
\in \mathbb{R},$ $\mathbf{1}_{d}$ is the vector with entries $1$ and $%
L_{t}(X^{x}),$ $0\leq t\leq T$ is the local time of the unknown solution
process, which one can define as

\begin{equation*}
L_{t}(X^{x})=\lim_{\varepsilon \searrow 0}\int_{0}^{t}\varphi _{\varepsilon
}(X_{s}^{x})ds
\end{equation*}%

where the limit is in probability and $\varphi _{\varepsilon }$ approximates, in distribution, the
Dirac delta function $\delta _{0}$ in zero. Here a commonly
used approximation $\varphi _{\varepsilon }$ is given by%

\begin{equation}
\varphi _{\varepsilon }(x)=\varepsilon ^{-\frac{d}{2}}\varphi (\varepsilon
^{-\frac{1}{2}}x),\text{ }\varepsilon >0,  \label{Dirac}
\end{equation}%

where $\varphi $ is a $d-$dimensional Gaussian probability density.

In the Wiener case, that is $H=\frac{1}{2}$, and $d=1$ solutions to
equations of this type are referred to as Skew Brownian motion in the
literature and were first studied by \cite{IM} and \cite{W} in the weak and
strong sense. See also the related articles \cite{HS}, \cite{Yor}, \cite%
{Revuz}, \cite{BassChen}, \cite{FRW} and \cite{Lejay}. In the sequel, we may
therefore also call solutions to (\ref{SkewfBm}) Skew Fractional Brownian
motions.

In the case of $H\in (0,1/2)$ the authors in \cite{BLPP} recently
constructed strong solutions to (\ref{SkewfBm}) by using techniques from
Malliavin calculus. In fact, the authors prove the following result:

\begin{theorem}
\label{Existence}Let $H<\frac{1}{2(d+2)}.$ Then for all $x\in \mathbb{R}^{d}$
and $\alpha \in \mathbb{R}$, there exists a strong solution to $X_{t}^{x},$ $%
0\leq t\leq T$ to (\ref{SkewfBm}). Moreover, $X_{t}^{x}$ is Malliavin
differentiable for all $0\leq t\leq T$.
\end{theorem}

\bigskip

For general $H\in (0,1)$ we shall mention the striking work in \cite{CG},
where the authors analyze \emph{path by path} solutions to (\ref{SkewfBm})
in the framework of Besov spaces $B_{\infty ,\infty }^{\alpha +1}$ by
employing techniques based e.g. on the Leray-Schauder-Tychonoff fixed point
theorem and a comparison principle with respect to an averaging operator. For $H<\frac{1}{2(d+1)}$, which is a slightly looser upper bound
for $H$ than the one in Theorem \ref{Existence}, the authors obtain existence of
strong solutions. In the case $H<\frac{1}{2(d+2)},$which corresponds to the
condition in Theorem \ref{Existence} they even prove path by path
uniqueness, but not Malliavin differentiability of such solutions. Further,
for $H<\frac{1}{2(d+3)}$ the authors are able to construct unique Lipschitz
flows. However, a disadvantage of the latter approach is that it, in contrast
to the method in \cite{BLPP}, cannot be used for the construction of strong
solutions to SDE's with additive fractional noise, where the drift vector
field belongs to $L^{\infty }(\mathbb{R}^{d})$.

The objective of this paper is to significantly improve the result in \cite%
{CG} with respect to the regularity of the stochastic flow in (\ref{SkewfBm}%
). For example in the case of $H<\frac{1}{2(d+3)}$, under which the authors
obtain a Lipschitz flow, we can show that the flow must be twice locally
Sobolev differentiable. Moreover, we will prove for $H<\frac{1}{2(d-1+2k)}$
that the flow belongs to the Sobolev space $W_{loc}^{k,p}$ for all $p\geq 2$.

The method used in this article is based on "local time variational
calculus" techniques developed in the papers \cite{BLPP}, \cite{CHOP}, \cite%
{BNP}. See also \cite{MMNPZ}, \cite{MP}, \cite{FNP}, \cite{HP} in the case
of a (cylindrical) Wiener process or a L\'{e}vy process.

\section{Main Result}

\bigskip Using a "local time variational calculus" technique for fractional
Brownian motion developed in \cite{BLPP}, \cite{CHOP}, \cite{BNP} we aim at
proving in this section higher order differentiability of the stochastic
flow associated with the SDE (\ref{SkewfBm}).

\bigskip The main result of our paper is the following:

\begin{theorem}

\label{MainResult}Let $H<\frac{1}{2(d-1+2k)}$ for $k\in \mathbb{N}$ and $%
\mathcal{U}\subset \mathbb{R}^{d}$ be a bounded and open set. Further, let $%
X_{t}^{x},$ $0\leq t\leq T$ be the strong solution to (\ref{SkewfBm}) as
constructed in Theorem \ref{Existence}. Then the associated stochastic flow
with respect to (\ref{SkewfBm}) is $k-$times Sobolev differentiable on $%
\mathcal{U}$ $\mu -$a.e.. More precisely, for all $0\leq t\leq T$ 
\begin{equation*}
(x\longmapsto X_{t}^{x})\in \dbigcap\limits_{p\geq 2}L^{2}(\Omega ;W^{k,p}(%
\mathcal{U})).
\end{equation*}

\end{theorem}

\begin{remark}
Let us mention here that the regularity result in Theorem \ref{MainResult}
is a significant improvement of that obtained in \cite{CG} in the case of
(distributional) drift vector fields in Besov spaces $B_{\infty ,\infty
}^{\alpha +1}$ for $\alpha >2-1/(2H),$ where the authors prove
Lipschitzianity of the associated stochastic flow.
\end{remark}

\bigskip

In order to prove Theorem \ref{MainResult} we need a some definition and an
auxiliary result:

\begin{definition}
\label{Class}\bigskip\ Let $H<\frac{1}{2(d+2)}$. We then denote by $\mathbb{L%
}$ the class of sequences of vector fields $\varphi _{n}:[0,T]\times \mathbb{%
R}^{d}\longrightarrow \mathbb{R},$ $n\geq 1$ such that the SDE 
\begin{equation}
Y_{t}^{x}=y+\int_{0}^{t}\varphi _{n}(u,Y_{u}^{y})\cdot \mathbf{1}%
_{d}du+B_{t}^{H},\text{ }0\leq t\leq T  \label{SDEPhi}
\end{equation}%

admits a unique strong solution for all $y\in \mathbb{R}^{d}$, $n\geq 1$ and
such that%

\begin{equation*}
\int_{0}^{\cdot }\varphi _{n}(u,B_{u}^{H}+y)du\in I_{0^{+}}^{H+\frac{1}{2}%
}(L^{2})
\end{equation*}%

for all $y\in \mathbb{R}^{d}$, $n\geq 1$ as well as%

\begin{equation*}
\sup_{n\geq 1}E\left[ \exp (k\int_{0}^{T}(K_{H}^{-1}(\int_{0}^{\cdot
}\varphi _{n}(u,B_{u}^{H}+y)du)(s))^{2}ds)\right] <\infty 
\end{equation*}%

for all $k\in \mathbb{R}$. See the Appendix for the definition of the space $%
I_{0^{+}}^{H+\frac{1}{2}}(L^{2})$ and the operator $K_{H}^{-1}$.

\end{definition}

\begin{remark}
\label{RemarkPhi} It follows from Lemma \ref{expmom} in the Appendix that
the approximation sequence $\varphi _{x,\varepsilon }$ for $\varepsilon
=1/n,n\geq 1$ with respect to the Dirac delta function in $x$ in (\ref%
{DiracApproximation}) belongs to the class $\mathbb{L}$.
\end{remark}

\bigskip

The proof of Theorem \ref{MainResult} mainly relies on the following
estimate (compare Lemma 7 in \cite{CHOP} and Theorem 5.1 in \cite{BNP}):

\begin{lemma}
\bigskip \label{HigherOrderDerivative}Assume that $H<\frac{1}{2(d-1+2k)}$.
Let $\varphi _{n}:\mathbb{R}^{d}\longrightarrow \mathbb{R},$ $n\geq 1$
belong to $\mathbb{L}$ in Definition \ref{Class} and $\varphi _{n}\in 
\mathcal{S}(\mathbb{R}^{d})$ (Schwartz function space) for all $n\geq 1$.
Denote by $X_{t}^{x,n},$ $0\leq t\leq T$ the strong solution to the SDE (\ref%
{SDEPhi}) with respect to the drift $\varphi _{n}$ for each $n\geq 1$. Fix
integers $p\geq 2$. Then 

\begin{equation*}
\sup_{x\in \mathbb{R}^{d}}E\left[ \left\Vert \frac{\partial ^{k}}{\partial
x^{k}}X_{t}^{x,n}\right\Vert ^{p}\right] \leq M\cdot
C_{p,k,H,d,T}(\left\Vert \varphi _{n}\right\Vert _{L^{1}(\mathbb{R}%
^{d})})<\infty 
\end{equation*}%

for all $n\geq 1$ for some continuous function $C_{p,k,H,d,T}:[0,\infty
)^{2}\longrightarrow \lbrack 0,\infty )$ and a constant $M$ depending only
on $\varphi _{n},$ $n\geq 1$ and $p$.
\end{lemma}

\begin{proof}
To simplify notation we set $b=\varphi _{n}\cdot \mathbf{1}_{d}$ for a fixed 
$n\geq 1$ and denote the corresponding solution by $X_{t}^{x}=X_{t}^{x,n},$ $%
0\leq t\leq T$. Since \ the stochastic flow associated with the smooth
vector field $b$ is also smooth (compare to e.g. \cite{Kunita}), we obtain
that%

\begin{equation}
\frac{\partial }{\partial x}X_{t}^{x}=I_{d\times
d}+\int_{s}^{t}Db(X_{u}^{x})\cdot \frac{\partial }{\partial x}X_{u}^{x}du,
\end{equation}%

where $Db:\mathbb{R}^{d}\longrightarrow L(\mathbb{R}^{d},\mathbb{R}^{d})$ is
the derivative of $b$ with respect to the space variable.

Using Picard iteration, we see that%

\begin{equation}
\frac{\partial }{\partial x}X_{t}^{x}=I_{d\times d}+\sum_{m\geq
1}\int_{\Delta
_{0,t}^{m}}Db(X_{u_{1}}^{x})\ldots Db(X_{u_{m}}^{x})du_{m}\ldots du_{1},
\label{FirstOrder}
\end{equation}%

where%

\begin{equation*}
\Delta _{s,t}^{m}=\{(u_{m},\ldots u_{1})\in \lbrack 0,T]^{m}:\theta
<u_{m}<\ldots <u_{1}<t\}.
\end{equation*}

By differentiating both sides with respect to $x$ in connection with
dominated convergence, we also get that%

\begin{equation*}
\frac{\partial ^{2}}{\partial x^{2}}X_{t}^{x}=\sum_{m\geq 1}\int_{\Delta
_{0,t}^{m}}\frac{\partial }{\partial x}%
[Db(X_{u_{1}}^{x})\ldots Db(X_{u_{m}}^{x})]du_{m}\ldots du_{1}.
\end{equation*}%

Using the Leibniz and chain rule, we have that%

\begin{eqnarray*}
\frac{\partial }{\partial x}[Db(X_{u_{1}}^{x})\ldots Db(X_{u_{m}}^{x})] = \sum_{r=1}^{m}Db(X_{u_{1}}^{x})\ldots D^{2}b(X_{u_{r}}^{x})\frac{\partial }{%
\partial x}X_{u_{r}}^{x}\ldots Db(X_{u_{m}}^{x}),
\end{eqnarray*}%
where $D^{2}b=D(Db):\mathbb{R}^{d}\longrightarrow L(\mathbb{R}^{d},L(\mathbb{%
R}^{d},\mathbb{R}^{d}))$.

\bigskip

So it follows from (\ref{FirstOrder}) that%

\begin{align}
\frac{\partial ^{2}}{\partial x^{2}}X_{t}^{x} =&\sum_{m_{1}\geq 1}\int_{\Delta
_{0,t}^{m_{1}}}\sum_{r=1}^{m_{1}}Db(X_{u_{1}}^{x})\ldots D^{2}b(X_{u_{r}}^{x}) 
\notag \\
&\times \left( I_{d\times d}+\sum_{m_{2}\geq 1}\int_{\Delta
_{0,u_{r}}^{m_{2}}}Db(X_{v_{1}}^{x})\ldots Db(X_{v_{m_{2}}}^{x}) dv_{m_{2}}\ldots  dv_{1}\right)
\notag \\
&\times Db(X_{u_{r+1}}^{x})\ldots  Db(X_{u_{m_{1}}}^{x}) du_{m_{1}}\ldots  du_{1} 
\notag \\
=&\sum_{m_{1}\geq 1}\sum_{r=1}^{m_{1}}\int_{\Delta
_{0,t}^{m_{1}}}Db(X_{u_{1}}^{x})\ldots  D^{2}b(X_{u_{r}}^{x})\ldots  Db(X_{u_{m_{1}}}^{x})du_{m_{1}}\ldots  du_{1}
\notag \\
&+\sum_{m_{1}\geq 1}\sum_{r=1}^{m_{1}}\sum_{m_{2}\geq 1}\int_{\Delta
_{0,t}^{m_{1}}}\int_{\Delta
_{0,u_{r}}^{m_{2}}}Db(X_{u_{1}}^{x})\ldots D^{2}b(X_{u_{r}}^{x})  \notag \\
&\times
Db(X_{v_{1}}^{x})\ldots Db(X_{v_{m_{2}}}^{x})Db(X_{u_{r+1}}^{x})\ldots Db(X_{u_{m_{1}}}^{x})
\notag \\
&dv_{m_{2}}\ldots dv_{1}du_{m_{1}}\ldots du_{1}  \notag \\
=&: I_{1}+I_{2}.  \label{SecondOrder}
\end{align}

\bigskip

We now aim at applying Lemma \ref{OrderDerivatives} (in connection with
Lemma \ref{partialshuffle}) to the term $I_{2}$ in (\ref{SecondOrder}) and
find that%

\begin{equation}
I_{2}=\sum_{m_{1}\geq 1}\sum_{r=1}^{m_{1}}\sum_{m_{2}\geq 1}\int_{\Delta
_{0,t}^{m_{1}+m_{2}}}\mathcal{H}%
_{m_{1}+m_{2}}^{X}(u)du_{m_{1}+m_{2}}\ldots du_{1}  \label{l2}
\end{equation}%

for $u=(u_{1},\ldots ,u_{m_{1}+m_{2}}),$ where the integrand $\mathcal{H}%
_{m_{1}+m_{2}}^{X}(u)\in \mathbb{R}^{d}\otimes \mathbb{R}^{d}\otimes \mathbb{%
R}^{d}$ has entries given by sums of at most $C(d)^{m_{1}+m_{2}}$ terms,
which are products of length $m_{1}+m_{2}$ of functions belonging to the set%

\begin{equation*}
\left\{ \frac{\partial ^{\gamma ^{(1)}+\ldots +\gamma ^{(d)}}}{\partial ^{\gamma
^{(1)}}x_{1}\ldots \partial ^{\gamma ^{(d)}}x_{d}}b^{(r)}(X_{u}^{x}),\text{ }%
r=1,\ldots ,d,\text{ }\gamma ^{(1)}+\ldots +\gamma ^{(d)}\leq 2,\text{ }\gamma
^{(l)}\in \mathbb{N}_{0},\text{ }l=1,\ldots ,d\right\} .
\end{equation*}%

Here it is important note that the terms in these products for which we have equality in \[\gamma ^{(1)}+\ldots +\gamma ^{(d)}\leq 2\] appear only once in (\ref{l2}). So the total order of derivatives $\left\vert \alpha
\right\vert $ (in the sense of Lemma \ref%
{OrderDerivatives} in the Appendix) of those products of functions is given by%

\begin{equation}
\left\vert \alpha \right\vert =m_{1}+m_{2}+1.
\end{equation}%

Let us choose $p,c,r\in \lbrack 1,\infty )$ such that $cp=2^{q}$ for some
integer $q$ and $\frac{1}{r}+\frac{1}{c}=1.$ Then we can use H\"{o}lder's
inequality and Girsanov's theorem (see Theorem \ref{girsanov}) in connection
with Definition \ref{Class} and obtain that%

\[
E[\left\Vert I_{2}\right\Vert ^{p}] 
\leq M\left( \sum_{m_{1}\geq 1}\sum_{r=1}^{m_{1}}\sum_{m_{2}\geq
1}\sum_{i\in I}\left\Vert \int_{\Delta _{0,t}^{m_{1}+m_{2}}}\mathcal{H}%
_{i}^{B^{H}}(u)du_{m_{1}+m_{2}}\ldots du_{1}\right\Vert _{L^{2^{q}}(\Omega ;%
\mathbb{R})}\right) ^{p},  \label{Lp}
\]

where $M<\infty $ is a constant depending only on $\varphi _{n},n\geq 1$ and 
$p$. Here $\#I\leq K^{m_{1}+m_{2}}$ for a constant $K=K(d)$ and the
integrands $\mathcal{H}_{i}^{B^{H}}(u)$ are of the form

\begin{equation*}
\mathcal{H}_{i}^{B^{H}}(u)=\dprod%
\limits_{l=1}^{m_{1}+m_{2}}h_{l}(u_{l}),h_{l}\in \Lambda ,l=1,\ldots ,m_{1}+m_{2}
\end{equation*}%

where 

\begin{equation*}
\Lambda :=\left\{ 
\begin{array}{c}
\frac{\partial ^{\gamma ^{(1)}+\ldots +\gamma ^{(d)}}}{\partial ^{\gamma
^{(1)}}x_{1}\ldots \partial ^{\gamma ^{(d)}}x_{d}}b^{(r)}(x+B_{u}^{H}),\text{ }%
r=1,\ldots ,d, \\ 
\gamma ^{(1)}+\ldots +\gamma ^{(d)}\leq 2,\text{ }\gamma ^{(l)}\in \mathbb{N}%
_{0},\text{ }l=1,\ldots ,d%
\end{array}%
\right\} .
\end{equation*}%

Also here functions with second order derivatives only appear once in those
products.

Set 

\begin{equation*}
J=\left( \int_{\Delta _{0,t}^{m_{1}+m_{2}}}\mathcal{H}%
_{i}^{B^{H}}(u)du_{m_{1}+m_{2}}\ldots du_{1}\right) ^{2^{q}}.
\end{equation*}%

Using Lemma \ref{partialshuffle} once more, successively $q$-times, we
find that $J$ can be written as a sum of length at most $%
K(q)^{m_{1}+m_{2}}$ with summands of the form%

\begin{equation}
\int_{\Delta
_{0,t}^{2^{q}(m_{1}+m_{2})}}\dprod%
\limits_{l=1}^{2^{q}(m_{1}+m_{2})}f_{l}(u_{l})du_{2^{q}(m_{1}+m_{2})}\ldots du_{1},
\label{f}
\end{equation}%

where $f_{l}\in \Lambda $ for all $l$.

Here the number of factors $f_{l}$ in the above product having a second
order derivative is exactly $2^{q}.$ So the total order of the derivatives
involved in (\ref{f}) in the sense of Lemma \ref{OrderDerivatives} (where
one in that lemma formally replaces $X_{u}^{x}$ by $x+B_{u}^{H}$ in the
corresponding expressions) is given by%
\begin{equation}
\left\vert \alpha \right\vert =2^{q}(m_{1}+m_{2}+1).  \label{alpha2}
\end{equation}

Now we can apply Theorem \ref{mainestimate2} for $m=2^{q}(m_{1}+m_{2})$ and $%
\varepsilon _{j}=0$ and get that%

\begin{eqnarray*}
&&\left\vert E\left[ \int_{\Delta
_{0,t}^{2^{q}(m_{1}+m_{2})}}\dprod%
\limits_{l=1}^{2^{q}(m_{1}+m_{2})}f_{l}(u_{l})du_{2^{q}(m_{1}+m_{2})}\ldots du_{1}%
\right] \right\vert  \\
&\leq &C^{m_{1}+m_{2}}(\left\Vert b\right\Vert _{L^{1}(\mathbb{R}%
^{d})})^{2^{q}(m_{1}+m_{2})} \\
&&\times \frac{((2(2^{q}(m_{1}+m_{2}+1))!)^{1/4}}{\Gamma
(-H(2d2^{q}(m_{1}+m_{2})+2^2 2^{q}(m_{1}+m_{2}+1))+22^{q}(m_{1}+m_{2}))^{1/2}}
\end{eqnarray*}%

\bigskip

for a constant $C$ depending on $H,T,d$ and $q$.

So the latter combined with (\ref{Lp}) shows that%

\begin{eqnarray*}
&&E[\left\Vert I_{2}\right\Vert ^{p}] \\
&\leq &M\left( \sum_{m_{1}\geq 1}\sum_{m_{2}\geq
1}K^{m_{1}+m_{2}}((\left\Vert b\right\Vert _{L^{1}(\mathbb{R}%
^{d})})^{2^{q}(m_{1}+m_{2})}\right.  \\
&&\left. \times \frac{((2(2^{q}(m_{1}+m_{2}+1))!)^{1/4}}{\Gamma
(-H(2d2^{q}(m_{1}+m_{2})+2^2 2^{q}(m_{1}+m_{2}+1))+2 2^{q}(m_{1}+m_{2}))^{1/2}}%
)^{1/2^{q}}\right) ^{p}
\end{eqnarray*}%

\bigskip

for a constant $K$ depending on $H,$ $T,$ $d,$ $p$ and $q$.

\bigskip

Since $\frac{1}{2(d+3)}\leq \frac{1}{2(d+2\frac{m_{1}+m_{2}+1}{m_{1}+m_{2}})}
$ for $m_{1},$ $m_{2}\geq 1$, the above sum converges if $H<\frac{1}{2(d+3)}$%
.

On the other hand one derives in the same way a similar estimate for $%
E[\left\Vert I_{1}\right\Vert ^{p}]$. Altogether the proof follows for $k=2$.

\bigskip

We now explain the generalization of the latter reasoning to the case $k\geq
2$: In this case, we find that%

\begin{equation}
\frac{\partial ^{k}}{\partial x^{k}}X_{t}^{x}=I_{1}+\ldots +I_{2^{k-1}},
\label{Ik}
\end{equation}%

where each $I_{i},$ $i=1,\ldots ,2^{k-1}$ is a sum of iterated integrals over
simplices of the form $\Delta _{0,u}^{m_{j}},$ $0<u<t,$ $j=1,\ldots ,k$ with
integrands having at most one product factor $D^{k}b$, whereas the other
factors are of the form $D^{j}b, j\leq k-1$.

For convenience, we introduce the following notation: For given
multi-indices $m.=(m_{1},\ldots ,m_{k})$ and $r:=(r_{1},\ldots ,r_{k-1})$ we define%

\begin{equation*}
m_{j}^{-}:=\sum_{i=1}^{j}m_{i}\text{ }
\end{equation*}%

and%

\begin{equation*}
\sum_{\substack{ m\geq 1 \\ r_{l}\leq m_{l}^{-} \\ l=1,\ldots ,k-1}}%
:=\sum_{m_{1}\geq 1}\sum_{r_{1}=1}^{m_{1}}\sum_{m_{2}\geq
1}\sum_{r_{2}=1}^{m_{2}^{-}}\ldots \sum_{r_{k-1}=1}^{m_{k-1}^{-}}\sum_{m_{k}\geq
1}.
\end{equation*}%

In what follows, we confine ourselves without loss of generality to the
estimation of the term $I_{2^{k-1}}$ in (\ref{Ik}). Just as in the case $k=2,
$ we obtain by employing Lemma \ref{OrderDerivatives} (in connection with
Lemma \ref{partialshuffle}) that 

\begin{equation}
I_{2^{k-1}}=\sum_{\substack{ m\geq 1 \\ r_{l}\leq m_{l}^{-} \\ l=1,\ldots ,k-1}}%
\int_{\Delta _{0,t}^{m_{1}+\ldots +m_{k}}}\mathcal{H}%
_{m_{1}+\ldots +m_{k}}^{X}(u)du_{m_{1}+m_{2}}\ldots du_{1}
\end{equation}%

for $u=(u_{m_{1}+\ldots +m_{k}},\ldots ,u_{1}),$ where the integrand $\mathcal{H}%
_{m_{1}+\ldots +m_{k}}^{X}(u)\in \otimes _{j=1}^{k+1}\mathbb{R}^{d}$ has entries
given by sums of at most $C(d)^{m_{1}+\ldots +m_{k}}$ terms, which are products
of length $m_{1}+\ldots + m_{k}$ of functions, which are elements in%

\begin{equation*}
\left\{ 
\begin{array}{c}
\frac{\partial ^{\gamma ^{(1)}+\ldots +\gamma ^{(d)}}}{\partial ^{\gamma
^{(1)}}x_{1}\ldots \partial ^{\gamma ^{(d)}}x_{d}}b^{(r)}(X_{u}^{x}),r=1,\ldots ,d,
\\ 
\gamma ^{(1)}+\ldots +\gamma ^{(d)}\leq k,\gamma ^{(l)}\in \mathbb{N}%
_{0},l=1,\ldots ,d%
\end{array}%
\right\} .
\end{equation*}%

As in the case $k=2$ we can apply Lemma \ref{OrderDerivatives} in the
Appendix and find that the total order of derivatives $\left\vert \alpha
\right\vert $ of those products of functions is 

\begin{equation}
\left\vert \alpha \right\vert =m_{1}+\ldots + m_{k}+k-1.
\end{equation}%

Then we proceed as before and choose $p,c,r\in \lbrack 1,\infty )$ such that 
$cp=2^{q}$ for some integer $q$ and $\frac{1}{r}+\frac{1}{c}=1$ and get by
using H\"{o}lder's inequality and Girsanov's theorem (see Theorem \ref%
{girsanov}) in connection with Definition \ref{Class} that%

\begin{eqnarray}
&&E[\left\Vert I_{2^{k-1}}\right\Vert ^{p}]  \notag \\
&\leq &M\left( \sum_{\substack{ m\geq 1 \\ r_{l}\leq m_{l}^{-} \\ l=1,\ldots ,k-1
}}\sum_{i\in I}\left\Vert \int_{\Delta _{0,t}^{m_{1}+m_{2}}}\mathcal{H}%
_{i}^{B^{H}}(u)du_{m_{1}+\ldots +m_{k}}\ldots du_{1}\right\Vert _{L^{2^{q}}(\Omega ;%
\mathbb{R})}\right) ^{p}  \label{Lp2}
\end{eqnarray}%

where $M<\infty $ is a constant depending only on $\varphi _{n},n\geq 1$ and 
$p$. Here $\#I\leq K^{m_{1}+\ldots +m_{k}}$ for a constant $K=K(d)$ and the
integrands $\mathcal{H}_{i}^{B^{H}}(u)$ take the form 

\begin{equation*}
\mathcal{H}_{i}^{B^{H}}(u)=\dprod\limits_{l=1}^{m_{1}+\ldots +m_{k}}h_{l}(u_{l}),%
\text{ }h_{l}\in \Lambda ,\text{ }l=1,\ldots ,m_{1}+\ldots +m_{k},
\end{equation*}%

where 

\begin{equation*}
\Lambda :=\left\{ 
\begin{array}{c}
\frac{\partial ^{\gamma ^{(1)}+\ldots +\gamma ^{(d)}}}{\partial ^{\gamma
^{(1)}}x_{1}\ldots \partial ^{\gamma ^{(d)}}x_{d}}b^{(r)}(x+B_{u}^{H}),\text{ }%
r=1,\ldots ,d, \\ 
\gamma ^{(1)}+\ldots +\gamma ^{(d)}\leq k,\text{ }\gamma ^{(l)}\in \mathbb{N}%
_{0},\text{ }l=1,\ldots ,d%
\end{array}%
\right\} .
\end{equation*}%

Let 

\begin{equation*}
J=\left( \int_{\Delta _{0,t}^{m_{1}+\ldots +m_{k}}}\mathcal{H}%
_{i}^{B^{H}}(u)du_{m_{1}+\ldots +m_{k}}\ldots du_{1}\right) ^{2^{q}}.
\end{equation*}%

Again repeated use of Lemma \ref{partialshuffle} in the Appendix shows that $%
J$ can be represented as a sum of, at most of length $K(q)^{m_{1}+\ldots .m_{k}}$
with summands of the form%

\begin{equation}
\int_{\Delta
_{0,t}^{2^{q}(m_{1}+\ldots +m_{k})}}\dprod%
\limits_{l=1}^{2^{q}(m_{1}+\ldots +m_{k})}f_{l}(u_{l})du_{2^{q}(m_{1}+\ldots .+m_{k})}\ldots du_{1},
\label{f2}
\end{equation}%

where $f_{l}\in \Lambda $ for all $l$.

Using once more Lemma \ref{OrderDerivatives} (where one in that Lemma
formally replaces $X_{u}^{x}$ by $x+B_{u}^{H}$ in the corresponding terms)
it follows that the total order of the derivatives in the products of
functions in (\ref{f2}) is given by%

\begin{equation}
\left\vert \alpha \right\vert =2^{q}(m_{1}+\ldots +m_{k}+k-1).
\end{equation}

Then Proposition \ref{mainestimate2} for $m=2^{q}(m_{1}+\ldots +m_{k})$ and $%
\varepsilon _{j}=0$ yields%

\begin{eqnarray*}
&&\left\vert E\left[ \int_{\Delta
_{0,t}^{2^{q}(m_{1}+\ldots +m_{k})}}\dprod%
\limits_{l=1}^{2^{q}(m_{1}+\ldots +m_{k})}f_{l}(u_{l})du_{2^{q}(m_{1}+\ldots +m_{k})}\ldots du_{1}%
\right] \right\vert  \\
&\leq &C^{m_{1}+\ldots +m_{k}}(\left\Vert b\right\Vert _{L^{1}(\mathbb{R}%
^{d})})^{2^{q}(m_{1}+\ldots +m_{k})} \\
&&\times \frac{((2(2^{q}(m_{1}+\ldots +m_{k}+k-1))!)^{1/4}}{\Gamma
(-H(2d2^{q}(m_{1}+\ldots +m_{k})+2^2 2^{q}(m_{1}+\ldots +m_{k}+k-1))+22^{q}(m_{1}+\ldots +m_{k}))^{1/2}%
}
\end{eqnarray*}%

for a constant $C$ depending on $H,$ $T,$ $d$ and $q$.

Hence (\ref{Lp2}) implies that%

\begin{eqnarray*}
&&E[\left\Vert I_{2^{k-1}}\right\Vert ^{p}] \\
&\leq &M\left( \sum_{m_{1}\geq 1}\ldots \sum_{m_{k}\geq
1}K^{m_{1}+\ldots +m_{k}}((\left\Vert b\right\Vert _{L^{1}(\mathbb{R}%
^{d})})^{2^{q}(m_{1}+\ldots +m_{k})}\right.  \\
&&\left. \times \frac{((2(2^{q}(m_{1}+\ldots +m_{k}+k-1))!)^{1/4}}{\Gamma
(-H(2d2^{q}(m_{1}+\ldots +m_{k})+2^2 2^{q}(m_{1}+\ldots +m_{k}+k-1))+22^{q}(m_{1}+\ldots +m_{k}))^{1/2}%
})^{1/2^{q}}\right) ^{p} \\
&\leq &M\left( \sum_{m\geq 1}\sum_{\substack{ l_{1},\ldots ,l_{k}\geq 0: \\ %
l_{1}+\ldots +l_{k}=m}}K^{m}((\left\Vert b\right\Vert _{L^{1}(\mathbb{R}%
^{d})})^{2^{q}m}\right.  \\
&&\left. \times \frac{((2(2^{q}(m+k-1))!)^{1/4}}{\Gamma
(-H(2d2^{q}m+2^2 2^{q}(m+k-1))+22^{q}m)^{1/2}})^{1/2^{q}}\right) ^{p}
\end{eqnarray*}%

for a constant $K$ depending on $H,$ $T,$ $d,$ $p$ and $q$.

Since we assumed that $H<$ $\frac{1}{2(d-1+2k)}$ the above sum converges. So
the proof follows.
\end{proof}

Using Lemma \ref{HigherOrderDerivative} we can now prove the main result.

\bigskip

\begin{proof}[Proof of Theorem \protect\ref{MainResult}]
Following the ideas in Theorem 5.2 in \cite{BNP} or Proposition 4.2 in \cite%
{MMNPZ}, we approximate the Dirac distribution $\delta _{0}$ in zero by $%
\varphi _{1/n},$ where $\varphi _{\varepsilon }\in \mathcal{S}(\mathbb{R}%
^{d}),$ $\varepsilon >0$ is given as in (\ref{Dirac}). Set $b_{n}=\varphi
_{1/n}\cdot \mathbf{1}_{d}$. Denote by $X_{t}^{n,x},$ $0\leq t\leq T$ the
solution to (\ref{SDEPhi}) associated with the vector field $b_{n}$,
starting in $x$. Let $\phi \in C_{c}^{\infty }(\mathcal{U};\mathbb{R}^{d})$
and define for fixed $t\in \lbrack 0,T]$ the sequence of random variables%

\begin{equation*}
\left\langle X_{t}^{n,\cdot },\phi \right\rangle :=\int_{\mathcal{U}%
}\left\langle X_{t}^{n,x},\phi \right\rangle _{\mathbb{R}^{d}}dx,\text{ }%
n\geq 1
\end{equation*}%

Using the same reasoning as in the proof of Theorem 5.2 in \cite{BNP}, which
is based on a compactness criterion for square integrable funtionals of
Wiener processes (see \cite{DMN}), combined with the estimates of Lemma 5.6
in \cite{BLPP} one shows that there exists a subsequence $n_{j},$ $j\geq 1$
such that%

\begin{equation}
\left\langle X_{t}^{n_{j},\cdot },\phi \right\rangle \underset{%
j\longrightarrow \infty }{\longrightarrow }\left\langle X_{t}^{\cdot },\phi
\right\rangle   \label{nj}
\end{equation}%

in $L^{2}(\Omega )$ strongly for all $\phi \in C_{c}^{\infty }(\mathcal{U};%
\mathbb{R}^{d})$, where $X_{s}^{x},$ $0\leq s\leq T$ is the strong solution
of Theorem \ref{Existence}. Note that we also have that%

\begin{equation*}
X_{t}^{n,x}\underset{n\longrightarrow \infty }{\longrightarrow }X_{t}^{x}
\end{equation*}%

in $L^{2}(\Omega )$ strongly.\ See Corollary 5.7 in \cite{BLPP}.

On the other hand, it follows from Lemma \ref{HigherOrderDerivative} that%

\begin{eqnarray*}
&&\sup_{n\geq 1}\left\Vert X_{t}^{n,\cdot }\right\Vert _{L^{2}(\Omega
;W^{k,p}(\mathcal{U}))}^{2} \\
&\leq &\sum_{i=0}^{k}\left( \int_{\mathcal{U}}\sup_{n\geq 1}E\left[
\left\Vert D^{i}X_{t}^{n,x}\right\Vert ^{p}\right] dx\right) ^{\frac{2}{p}%
}<\infty 
\end{eqnarray*}%

for $H<\frac{1}{2(d-1+2k)}.$

Since $L^{2}(\Omega ;W^{k,p}(\mathcal{U}))$ is reflexive for $p>1,$ there
exists a subsequence $n_{j},j\geq 1$ such that%

\begin{equation*}
X_{t}^{n,\cdot }\underset{j\longrightarrow \infty }{\longrightarrow }Y
\end{equation*}%

in $L^{2}(\Omega ;W^{k,p}(\mathcal{U}))$ weakly. For convenience, let $%
n_{j},j\geq 1$ be the same subsequence as in (\ref{nj}). Further, we have
for all $A\in \mathcal{F},$ $\phi \in C_{c}^{\infty }(\mathcal{U};\mathbb{R}%
^{d}),$ $\alpha ^{(1)}+\ldots +\alpha ^{(d)}\leq k$ with $\alpha ^{(i)}\in 
\mathbb{N}_{0},$ $i=1,\ldots ,d$ that%

\begin{eqnarray*}
&&E\left[ 1_{A}\left\langle X_{t}^{n_{j},\cdot },\frac{\partial ^{\alpha
^{(1)}+\ldots +\alpha ^{(d)}}}{\partial ^{\alpha ^{(1)}}x_{1}\ldots \partial
^{\alpha ^{(d)}}x_{d}}\phi \right\rangle \right]  \\
&=&(-1)^{\alpha ^{(1)}+\ldots +\alpha ^{(d)}}E\left[ 1_{A}\left\langle \frac{%
\partial ^{\alpha ^{(1)}+\ldots +\alpha ^{(d)}}}{\partial ^{\alpha
^{(1)}}x_{1}\ldots \partial ^{\alpha ^{(d)}}x_{d}}X_{t}^{n_{j},\cdot },\phi
\right\rangle \right]  \\
&&\underset{j\longrightarrow \infty }{\longrightarrow }(-1)^{\alpha
^{(1)}+\ldots +\alpha ^{(d)}}E\left[ 1_{A}\left\langle \frac{\partial ^{\alpha
^{(1)}+\ldots +\alpha ^{(d)}}}{\partial ^{\alpha ^{(1)}}x_{1}\ldots \partial
^{\alpha ^{(d)}}x_{d}}Y,\phi \right\rangle \right] .
\end{eqnarray*}%
We also know from (\ref{nj}) that%
\begin{equation*}
E\left[ 1_{A}\left\langle X_{t}^{n_{j},\cdot },\frac{\partial ^{\alpha
^{(1)}+\ldots +\alpha ^{(d)}}}{\partial ^{\alpha ^{(1)}}x_{1}\ldots \partial
^{\alpha ^{(d)}}x_{d}}\phi \right\rangle \right] \underset{j\longrightarrow
\infty }{\longrightarrow }E\left[ 1_{A}\left\langle X_{t}^{\cdot },\frac{%
\partial ^{\alpha ^{(1)}+\ldots +\alpha ^{(d)}}}{\partial ^{\alpha
^{(1)}}x_{1}\ldots \partial ^{\alpha ^{(d)}}x_{d}}\phi \right\rangle \right] .
\end{equation*}%

So $X_{t}^{\cdot }\in L^{2}(\Omega ;W^{k,p}(\mathcal{U}))$ for all $p\geq 2$.
\end{proof}

\section{Appendix}

In view of the need for a version of Girsanov's theorem for fractional Brownian
motion, which we use in connection with the proof of Lemma \ref%
{HigherOrderDerivative}, we recall some basic concepts from fractional
calculus (see \cite{samko.et.al.93} and \cite{lizorkin.01}).

Let $a,$ $b\in \mathbb{R}$ with $a<b$. Let $f\in L^{p}([a,b])$ with $p\geq 1$
and $\alpha >0$. Define the \emph{left-} and \emph{right-sided
Riemann-Liouville fractional integrals} by 

\begin{equation*}
I_{a^{+}}^{\alpha }f(x)=\frac{1}{\Gamma (\alpha )}\int_{a}^{x}(x-y)^{\alpha
-1}f(y)dy
\end{equation*}%

and 

\begin{equation*}
I_{b^{-}}^{\alpha }f(x)=\frac{1}{\Gamma (\alpha )}\int_{x}^{b}(y-x)^{\alpha
-1}f(y)dy
\end{equation*}%

for almost all $x\in \lbrack a,b]$, where $\Gamma $ denotes the Gamma
function.

For a given integer $p\geq 1$, let $I_{a^{+}}^{\alpha }(L^{p})$ (resp. $%
I_{b^{-}}^{\alpha }(L^{p})$) be the image of $L^{p}([a,b])$ of the operator $%
I_{a^{+}}^{\alpha }$ (resp. $I_{b^{-}}^{\alpha }$). If $f\in
I_{a^{+}}^{\alpha }(L^{p})$ (resp. $f\in I_{b^{-}}^{\alpha }(L^{p})$) and $%
0<\alpha <1$ then we can introduce the \emph{left-} and \emph{right-sided
Riemann-Liouville fractional derivatives} by 

\begin{equation*}
D_{a^{+}}^{\alpha }f(x)=\frac{1}{\Gamma (1-\alpha )}\frac{d}{dx}\int_{a}^{x}%
\frac{f(y)}{(x-y)^{\alpha }}dy
\end{equation*}%

and 

\begin{equation*}
D_{b^{-}}^{\alpha }f(x)=\frac{1}{\Gamma (1-\alpha )}\frac{d}{dx}\int_{x}^{b}%
\frac{f(y)}{(y-x)^{\alpha }}dy.
\end{equation*}

The left- and right-sided derivatives of $f$ also have the following
representations 

\begin{equation*}
D_{a^{+}}^{\alpha }f(x)=\frac{1}{\Gamma (1-\alpha )}\left( \frac{f(x)}{%
(x-a)^{\alpha }}+\alpha \int_{a}^{x}\frac{f(x)-f(y)}{(x-y)^{\alpha +1}}%
dy\right)
\end{equation*}%

and
 
\begin{equation*}
D_{b^{-}}^{\alpha }f(x)=\frac{1}{\Gamma (1-\alpha )}\left( \frac{f(x)}{%
(b-x)^{\alpha }}+\alpha \int_{x}^{b}\frac{f(x)-f(y)}{(y-x)^{\alpha +1}}%
dy\right) .
\end{equation*}

Using the above definitions, one finds that 
\begin{equation*}
I_{a^{+}}^{\alpha }(D_{a^{+}}^{\alpha }f)=f
\end{equation*}%
for all $f\in I_{a^{+}}^{\alpha }(L^{p})$ and 
\begin{equation*}
D_{a^{+}}^{\alpha }(I_{a^{+}}^{\alpha }f)=f
\end{equation*}%
for all $f\in L^{p}([a,b])$ and similarly for $I_{b^{-}}^{\alpha }$ and $%
D_{b^{-}}^{\alpha }$.

\bigskip

Consider now a $d$-dimensional \emph{fractional Brownian motion} $%
B^{H}=\{B_{t}^{H},t\in \lbrack 0,T]\}$ with Hurst parameter $H\in (0,1/2)$,
that is $B^{H}$ is a centered Gaussian process with a covariance function
given by 

\begin{equation*}
(R_{H}(t,s))_{i,j}:=E[B_{t}^{H,(i)}B_{s}^{H,(j)}]=\delta _{ij}\frac{1}{2}%
\left( t^{2H}+s^{2H}-|t-s|^{2H}\right) ,\quad i,j=1,\dots ,d,
\end{equation*}%

where $\delta _{ij}$ is one, if $i=j$, or zero else.

Next, we want to briefly pass in review a construction of the fractional
Brownian motion, which can be found in \cite{Nualart}. For convenience let $%
d=1$.

Denote by $\mathcal{E}$ the set of step functions on $[0,T]$ and by $%
\mathcal{H}$ the Hilbert space given by the completion of $\mathcal{E}$
with respect to the inner product 

\begin{equation*}
\langle 1_{[0,t]},1_{[0,s]}\rangle _{\mathcal{H}}=R_{H}(t,s).
\end{equation*}%

From that we obtain an extension of the mapping $1_{[0,t]}\mapsto B_{t}$ to
an isometry between $\mathcal{H}$ and a Gaussian subspace of $L^{2}(\Omega )$
associated with $B^{H}$. Denote by $\varphi \mapsto B^{H}(\varphi )$ this
isometry.

It turns out that for $H<1/2$ the covariance function $R_{H}(t,s)$ can be
represented as

\bigskip\ 

\begin{equation}
R_{H}(t,s)=\int_{0}^{t\wedge s}K_{H}(t,u)K_{H}(s,u)du,  \label{2.2}
\end{equation}%

where 

\begin{equation*}
K_{H}(t,s)=c_{H}\left[ \left( \frac{t}{s}\right) ^{H-\frac{1}{2}}(t-s)^{H-%
\frac{1}{2}}+\left( \frac{1}{2}-H\right) s^{\frac{1}{2}-H}\int_{s}^{t}u^{H-%
\frac{3}{2}}(u-s)^{H-\frac{1}{2}}du\right] .
\end{equation*}%

Here $c_{H}=\sqrt{\frac{2H}{(1-2H)\beta (1-2H,H+1/2)}}$ and $\beta $ is the
Beta function. See \cite[Proposition 5.1.3]{Nualart}.

Using the kernel $K_{H}$, one can define by means (\ref{2.2}) an isometry $%
K_{H}^{\ast }$ between $\mathcal{E}$ and $L^{2}([0,T])$ such that $%
(K_{H}^{\ast }1_{[0,t]})(s)=K_{H}(t,s)1_{[0,t]}(s).$ This isometry extends
to the Hilbert space $\mathcal{H\ }$and has the following representations in
terms of fractional derivatives

\begin{equation*}
(K_{H}^{\ast }\varphi )(s)=c_{H}\Gamma \left( H+\frac{1}{2}\right) s^{\frac{1%
}{2}-H}\left( D_{T^{-}}^{\frac{1}{2}-H}u^{H-\frac{1}{2}}\varphi (u)\right)
(s)
\end{equation*}%

and 

\begin{align*}
(K_{H}^{\ast }\varphi )(s)=& \,c_{H}\Gamma \left( H+\frac{1}{2}\right)
\left( D_{T^{-}}^{\frac{1}{2}-H}\varphi (s)\right) (s) \\
& +c_{H}\left( \frac{1}{2}-H\right) \int_{s}^{T}\varphi (t)(t-s)^{H-\frac{3}{%
2}}\left( 1-\left( \frac{t}{s}\right) ^{H-\frac{1}{2}}\right) dt.
\end{align*}%

for $\varphi \in \mathcal{H}$ One also has that $\mathcal{H}=I_{T^{-}}^{%
\frac{1}{2}-H}(L^{2})$. See \cite{DU} and \cite[Proposition 6]%
{alos.mazet.nualart.01}.

Using the fact that $K_{H}^{\ast }$ is an isometry from $\mathcal{H}$ into $%
L^{2}([0,T])$ the $d$-dimensional process $W=\{W_{t},t\in \lbrack 0,T]\}$
defined by 

\begin{equation}
W_{t}:=B^{H}((K_{H}^{\ast })^{-1}(1_{[0,t]}))  \label{WBH}
\end{equation}%

is a Wiener process and the process $B^{H}$ has the representation 

\begin{equation}
B_{t}^{H}=\int_{0}^{t}K_{H}(t,s)dW_{s}.  \label{BHW}
\end{equation}%

See \cite{alos.mazet.nualart.01}.

In the sequel we also need the concept of fractional Brownian motion with
respect to a filtration.

\begin{definition}
Let $\mathcal{G}=\left\{ \mathcal{G}_{t}\right\} _{t\in \left[ 0,T\right] }$
be a filtration on $\left( \Omega ,\mathcal{F},P\right) $ satisfy the usual
conditions. A fractional Brownian motion $B^{H}$ is called a $\mathcal{G}$%
-fractional Brownian motion if the process $W$ defined by (\ref{WBH}) is a $%
\mathcal{G}$-Brownian motion.
\end{definition}

\bigskip

In what follows, let $W$ be a standard Wiener process on a probability space 
$(\Omega ,\mathfrak{A},P)$ endowed with the natural filtration $\mathcal{F}%
=\{\mathcal{F}_{t}\}_{t\in \lbrack 0,T]}$ which is generated by $W$ and
augmented by all $P$-null sets.\ We denote by $B:=B^{H}$ the fractional
Brownian motion with Hurst parameter $H\in (0,1/2)$ given by the
representation (\ref{BHW}).

We want to employ a version of Girsanov's theorem for fractional Brownian
motion which goes back to \cite[Theorem 4.9]{DU}. The version we recall here
is that given in \cite[Theorem 2]{NO}. In doing so, we have to introduce the
definition of an isomorphism $K_{H}$ from $L^{2}([0,T])$ onto $I_{0+}^{H+%
\frac{1}{2}}(L^{2})$ associated with the kernel $K_{H}(t,s)$ in terms of the
fractional integrals as follows (see \cite[Theorem 2.1]{DU}): 
\begin{equation*}
(K_{H}\varphi )(s)=I_{0^{+}}^{2H}s^{\frac{1}{2}-H}I_{0^{+}}^{\frac{1}{2}%
-H}s^{H-\frac{1}{2}}\varphi ,\quad \varphi \in L^{2}([0,T]).
\end{equation*}

From that and the properties of the Riemann-Liouville fractional integrals
and derivatives one can see that the inverse of $K_{H}$ has the
representation 

\begin{equation*}
(K_{H}^{-1}\varphi )(s)=s^{\frac{1}{2}-H}D_{0^{+}}^{\frac{1}{2}-H}s^{H-\frac{%
1}{2}}D_{0^{+}}^{2H}\varphi (s),\quad \varphi \in I_{0+}^{H+\frac{1}{2}%
}(L^{2}).
\end{equation*}

The latter shows that if $\varphi $ is absolutely continuous, see \cite{NO},
one gets that 
\begin{equation*}
(K_{H}^{-1}\varphi )(s)=s^{H-\frac{1}{2}}I_{0^{+}}^{\frac{1}{2}-H}s^{\frac{1%
}{2}-H}\varphi ^{\prime }(s).
\end{equation*}

\begin{theorem}[Girsanov's theorem for fBm]
\label{girsanov} Let $u=\{u_t, t\in [0,T]\}$ be an $\mathcal{F}$-adapted
process with integrable trajectories and set $\widetilde{B}_t^H = B_t^H +
\int_0^t u_s ds, \quad t\in [0,T].$ Assume that

\begin{itemize}
\item[(i)] $\int_{0}^{\cdot }u_{s}ds\in I_{0+}^{H+\frac{1}{2}}(L^{2}([0,T]))$%
, $P$-a.s.

\item[(ii)] $E[\xi_T]=1$ where 
\begin{equation*}
\xi_T := \exp\left\{-\int_0^T K_H^{-1}\left( \int_0^{\cdot} u_r
dr\right)(s)dW_s - \frac{1}{2} \int_0^T K_H^{-1} \left( \int_0^{\cdot} u_r
dr \right)^2(s)ds \right\}.
\end{equation*}
\end{itemize}

Then the shifted process $\widetilde{B}^H$ is an $\mathcal{F}$-fractional
Brownian motion with Hurst parameter $H$ under the new probability $%
\widetilde{P}$ defined by $\frac{d\widetilde{P}}{dP}=\xi_T$.
\end{theorem}

\begin{remark}
As for the multi-dimensional case, define 
\begin{equation*}
(K_{H}\varphi )(s):=((K_{H}\varphi ^{(1)})(s),\dots ,(K_{H}\varphi
^{(d)})(s))^{\ast },\quad \varphi \in L^{2}([0,T];\mathbb{R}^{d}),
\end{equation*}%
where $\ast $ denotes transposition. Similarly for $K_{H}^{-1}$ and $%
K_{H}^{\ast }$.
\end{remark}

In this paper we also make use of the following technical lemma, whose proof
can be found in \cite[Lemma 5.3]{BLPP}

\begin{lemma}
\label{expmom} Let $x\in \mathbb{R}^{d}$. If $H<\frac{1}{2(1+d)}$ then%
\begin{equation*}
\sup_{\varepsilon >0}E\left[ \exp
(k\int_{0}^{T}(K_{H}^{-1}(\int_{0}^{.}\varphi _{x,\varepsilon
}(B_{u}^{H})du)(t))^{2}dt)\right] <\infty 
\end{equation*}%
for all $k\in \mathbb{R},$ where%
\begin{equation}
\varphi _{x,\varepsilon }(B_{u}^{H})=\frac{1}{(2\pi \varepsilon )^{\frac{d}{2%
}}}\exp (-\frac{\left\vert B_{u}^{H}-x\right\vert _{\mathbb{R}^{d}}^{2}}{%
2\varepsilon }).  \label{DiracApproximation}
\end{equation}
\end{lemma}

\bigskip

In the following we also need an integration by parts formula for iterated
integrals based on \emph{shuffle permutations}. For this purpose, let $m$
and $n$ be integers. We define $S(m,n)$ as the set of shuffle permutations,
i.e. the set of permutations $\sigma :\{1,\dots ,m+n\}\rightarrow \{1,\dots
,m+n\}$ such that $\sigma (1)<\dots <\sigma (m)$ and $\sigma (m+1)<\dots
<\sigma (m+n)$.

Define the $m$-dimensional simplex for $0\leq \theta <t\leq T$, 
\begin{equation*}
\Delta _{\theta ,t}^{m}:=\{(s_{m},\dots ,s_{1})\in \lbrack 0,T]^{m}:\,\theta
<s_{m}<\cdots <s_{1}<t\}.
\end{equation*}%
The product of two simplices can be written as the following union 
\begin{equation*}
\Delta _{\theta ,t}^{m}\times \Delta _{\theta ,t}^{n}=%
\mbox{\footnotesize
$\bigcup_{\sigma \in S(m,n)} \{(w_{m+n},\dots,w_1)\in [0,T]^{m+n} : \,
\theta< w_{\sigma(m+n)} <\cdots < w_{\sigma(1)} <t\} \cup \mathcal{N}$
\normalsize},
\end{equation*}%
where the set $\mathcal{N}$ has null Lebesgue measure. Hence, if $%
f_{i}:[0,T]\rightarrow \mathbb{R}$, $i=1,\dots ,m+n$ are integrable
functions we have 
\begin{align}
\int_{\Delta _{\theta ,t}^{m}}\prod_{j=1}^{m}f_{j}(s_{j})ds_{m}\dots ds_{1}&
\int_{\Delta _{\theta ,t}^{n}}\prod_{j=m+1}^{m+n}f_{j}(s_{j})ds_{m+n}\dots
ds_{m+1}  \notag \\
& =\sum_{\sigma \in S(m,n)}\int_{\Delta _{\theta
,t}^{m+n}}\prod_{j=1}^{m+n}f_{\sigma (j)}(w_{j})dw_{m+n}\cdots dw_{1}.
\label{shuffleIntegral}
\end{align}

The latter relation can be generalized as follows (see \cite{BNP}):

\begin{lemma}
\label{partialshuffle} Let $n,$ $p$ and $k$ be non-negative integers, $k\leq
n$. Assume we have integrable functions $f_{j}:[0,T]\rightarrow \mathbb{R}$, 
$j=1,\dots ,n$ and $g_{i}:[0,T]\rightarrow \mathbb{R}$, $i=1,\dots ,p$. We
may then write 
\begin{align*}
& \int_{\Delta _{\theta ,t}^{n}}f_{1}(s_{1})\dots f_{k}(s_{k})\int_{\Delta
_{\theta ,s_{k}}^{p}}g_{1}(r_{1})\dots g_{p}(r_{p})dr_{p}\dots
dr_{1}f_{k+1}(s_{k+1})\dots f_{n}(s_{n})ds_{n}\dots ds_{1} \\
& =\sum_{\sigma \in A_{n,p}}\int_{\Delta _{\theta ,t}^{n+p}}h_{1}^{\sigma
}(w_{1})\dots h_{n+p}^{\sigma }(w_{n+p})dw_{n+p}\dots dw_{1},
\end{align*}%
where $h_{l}^{\sigma }\in \{f_{j},g_{i}:1\leq j\leq n,1\leq i\leq p\}$.
Above $A_{n,p}$ denotes a subset of permutations of $\{1,\dots ,n+p\}$ such
that $\#A_{n,p}\leq C^{n+p}$ for an appropriate constant $C\geq 1$. Here we
defined $s_{0}=\theta $.
\end{lemma}

\bigskip

The proof of Lemma \ref{HigherOrderDerivative} requires an important
estimate (see e.g. Proposition 3.3 in \cite{BLPP} for a new proof).
To this end, let $m$ be an integer and let $f:[0,T]^{m}\times (\mathbb{R}%
^{d})^{m}\rightarrow \mathbb{R}$ be a function of the form 
\begin{equation}
f(s,z)=\prod_{j=1}^{m}f_{j}(s_{j},z_{j}),\quad s=(s_{1},\dots ,s_{m})\in
\lbrack 0,T]^{m},\quad z=(z_{1},\dots ,z_{m})\in (\mathbb{R}^{d})^{m},
\label{ff}
\end{equation}%
where $f_{j}:[0,T]\times \mathbb{R}^{d}\rightarrow \mathbb{R}$, $j=1,\dots ,m
$ are smooth functions with compact support. Further, let $\varkappa
:[0,T]^{m}\rightarrow \mathbb{R}$ be a function of the form 
\begin{equation}
\varkappa (s)=\prod_{j=1}^{m}\varkappa _{j}(s_{j}),\quad s\in \lbrack
0,T]^{m},  \label{kappa}
\end{equation}%
where $\varkappa _{j}:[0,T]\rightarrow \mathbb{R}$, $j=1,\dots ,m$ are
integrable functions.

Next, denote by $\alpha_j$ a multi-index and $D^{\alpha_j}$ its
corresponding differential operator. For $\alpha = (\alpha_1, \dots,
\alpha_m)$ considered an element of $\mathbb{N}_0^{d\times m}$ so that $%
|\alpha|:= \sum_{j=1}^m \sum_{l=1}^d \alpha_{j}^{(l)}$, we write 
\begin{equation*}
D^{\alpha}f(s,z) = \prod_{j=1}^m D^{\alpha_j} f_j(s_j,z_j).
\end{equation*}

\begin{theorem}
\label{mainestimate2} Let $B^{H},H\in (0,1/2)$ be a standard $d-$dimensional
fractional Brownian motion and functions $f$ and $\varkappa $ as in (\ref{ff}%
), respectively as in (\ref{kappa}). Let $\theta ,t\in \lbrack 0,T]$ with $%
\theta <t$ and%
\begin{equation*}
\varkappa _{j}(s)=(K_{H}(s,\theta ))^{\varepsilon _{j}},\theta <s<t
\end{equation*}%
for every $j=1,\ldots ,m$ with $(\varepsilon _{1},\ldots ,\varepsilon _{m})\in
\{0,1\}^{m}.$ Let $\alpha \in (\mathbb{N}_{0}^{d})^{m}$ be a multi-index. If 
\begin{equation*}
H<\frac{\frac{1}{2}-\gamma }{(d-1+2\sum_{l=1}^{d}\alpha _{j}^{(l)})}
\end{equation*}%
for all $j$, where $\gamma \in (0,H)$ is sufficiently small, then there
exists a universal constant $C$ (depending on $H$, $T$ and $d$, but
independent of $m$, $\{f_{i}\}_{i=1,\ldots ,m}$ and $\alpha $) such that for any 
$\theta ,t\in \lbrack 0,T]$ with $\theta <t$ we have%

\begin{eqnarray*}
&&\left\vert E\int_{\Delta _{\theta ,t}^{m}}\left( \prod_{j=1}^{m}D^{\alpha
_{j}}f_{j}(s_{j},B_{s_{j}}^{H})\varkappa _{j}(s_{j})\right) ds\right\vert  \\
&\leq &C^{m+\left\vert \alpha \right\vert }\prod_{j=1}^{m}\left\Vert
f_{j}(\cdot ,z_{j})\right\Vert _{L^{1}(\mathbb{R}^{d};L^{\infty
}([0,T]))}\theta ^{(H-\frac{1}{2})\sum_{j=1}^{m}\varepsilon _{j}} \\
&&\times \frac{(\prod_{l=1}^{d}(2\left\vert \alpha ^{(l)}\right\vert
)!)^{1/4}(t-\theta )^{-H(md+2\left\vert \alpha \right\vert )-(H-\frac{1}{2}%
-\gamma )\sum_{j=1}^{m}\varepsilon _{j}+m}}{\Gamma (-H(2md+4\left\vert
\alpha \right\vert )+2(H-\frac{1}{2}-\gamma )\sum_{j=1}^{m}\varepsilon
_{j}+2m)^{1/2}}.
\end{eqnarray*}
\end{theorem}

\begin{remark}
\label{S}The above theorem also holds true for functions $%
\{f_{i}\}_{i=1,\ldots ,m}$ in the Schwartz function space.
\end{remark}

\bigskip

Finally, we also need the following auxiliary result in connection with the
proof of Lemma \ref{HigherOrderDerivative}:

\begin{lemma}
\label{OrderDerivatives}Let $n,$ $p$ and $k$ be non-negative integers, $%
k\leq n$. Assume we have functions $f_{j}:[0,T]\rightarrow \mathbb{R}$, $%
j=1,\dots ,n$ and $g_{i}:[0,T]\rightarrow \mathbb{R}$, $i=1,\dots ,p$ such
that

\begin{equation*}
f_{j}\in \left\{ \frac{\partial ^{\alpha _{j}^{(1)}+\ldots +\alpha _{j}^{(d)}}}{%
\partial ^{\alpha _{j}^{(1)}}x_{1}\ldots \partial ^{\alpha _{j}^{(d)}}x_{d}}%
b^{(r)}(X_{u}^{x}),\text{ }r=1,\ldots ,d\right\} ,\text{ }j=1,\ldots ,n
\end{equation*}%

and

\begin{equation*}
g_{i}\in \left\{ \frac{\partial ^{\beta _{i}^{(1)}+\ldots +\beta _{i}^{(d)}}}{%
\partial ^{\beta _{i}^{(1)}}x_{1}\ldots \partial ^{\beta _{i}^{(d)}}x_{d}}%
b^{(r)}(X_{u}^{x}),\text{ }r=1,\ldots ,d\right\} ,\text{ }i=1,\ldots ,p
\end{equation*}%

for $\alpha :=(\alpha _{j}^{(l)})\in \mathbb{N}_{0}^{d\times n}$ and $\beta
:=(\beta _{i}^{(l)})\in \mathbb{N}_{0}^{d\times p},$ where $X_{\cdot }^{x}$
is the strong solution to 

\begin{equation*}
X_{t}^{x}=x+\int_{0}^{t}b(X_{u}^{x})du+B_{t}^{H},\text{ }0\leq t\leq T
\end{equation*}%

for $b=(b^{(1)},\ldots ,b^{(d)})$ with $b^{(r)}\in \mathcal{S}(\mathbb{R}^{d})$
for all $r=1,\ldots ,d$. So (as we shall say in the sequel) the product $%
g_{1}(r_{1})\cdot \dots \cdot g_{p}(r_{p})$ has a total order of derivatives 
$\left\vert \beta \right\vert =\sum_{l=1}^{d}\sum_{i=1}^{p}\beta _{i}^{(l)}$%
. We know from Lemma \ref{partialshuffle} that 

\begin{align}
& \int_{\Delta _{\theta ,t}^{n}}f_{1}(s_{1})\dots f_{k}(s_{k})\int_{\Delta
_{\theta ,s_{k}}^{p}}g_{1}(r_{1})\dots g_{p}(r_{p})dr_{p}\dots
dr_{1}f_{k+1}(s_{k+1})\dots f_{n}(s_{n})ds_{n}\dots ds_{1}  \notag \\
& =\sum_{\sigma \in A_{n,p}}\int_{\Delta _{\theta ,t}^{n+p}}h_{1}^{\sigma
}(w_{1})\dots h_{n+p}^{\sigma }(w_{n+p})dw_{n+p}\dots dw_{1},  \label{h}
\end{align}%

where $h_{l}^{\sigma }\in \{f_{j},g_{i}:1\leq j\leq n,$ $1\leq i\leq p\}$, $%
A_{n,p}$ is a subset of permutations of $\{1,\dots ,n+p\}$ such that $%
\#A_{n,p}\leq C^{n+p}$ for an appropriate constant $C\geq 1$, and $%
s_{0}=\theta $. Then the products%

\begin{equation*}
h_{1}^{\sigma }(w_{1})\cdot \dots \cdot h_{n+p}^{\sigma }(w_{n+p})
\end{equation*}%

have a total order of derivatives given by $\left\vert \alpha \right\vert
+\left\vert \beta \right\vert .$
\end{lemma}

\begin{proof}
The result is proved by induction on $n$. For $n=1$ and $k=0$ the result is
trivial. For $k=1$ we have 

\begin{eqnarray*}
\int_{\theta }^{t}f_{1}(s_{1})\int_{\Delta _{\theta
,s_{1}}^{p}}g_{1}(r_{1})\dots g_{p}(r_{p}) &&dr_{p}\dots dr_{1}ds_{1} \\
&=&\int_{\Delta _{\theta ,t}^{p+1}}f_{1}(w_{1})g_{1}(w_{2})\dots
g_{p}(w_{p+1})dw_{p+1}\dots dw_{1},
\end{eqnarray*}%

where we have put $w_{1}=s_{1},$ $w_{2}=r_{1},\dots ,w_{p+1}=r_{p}$. Hence
the total order of derivatives involved in the product of the last integral
is given by $\sum_{l=1}^{d}\alpha
_{1}^{(l)}+\sum_{l=1}^{d}\sum_{i=1}^{p}\beta _{i}^{(l)}=\left\vert \alpha
\right\vert +\left\vert \beta \right\vert .$

Assume the result holds for $n$ and let us show that this implies that the
result is true for $n+1$. Either $k=0,1$ or $2\leq k\leq n+1$. For $k=0$ the
result is trivial. For $k=1$ we have 

\begin{align*}
\int_{\Delta _{\theta ,t}^{n+1}}& f_{1}(s_{1})\int_{\Delta _{\theta
,s_{1}}^{p}}g_{1}(r_{1})\dots g_{p}(r_{p})dr_{p}\dots
dr_{1}f_{2}(s_{2})\dots f_{n+1}(s_{n+1})ds_{n+1}\dots ds_{1} \\
& =\int_{\theta }^{t}f_{1}(s_{1})\left( \int_{\Delta _{\theta
,s_{1}}^{n}}\int_{\Delta _{\theta ,s_{1}}^{p}}g_{1}(r_{1})\dots
g_{p}(r_{p})dr_{p}\dots dr_{1}f_{2}(s_{2})\dots
f_{n+1}(s_{n+1})ds_{n+1}\dots ds_{2}\right) ds_{1}.
\end{align*}%

From (\ref{shuffleIntegral}) we observe by using the shuffle permutations
that the latter inner double integral on diagonals can be written as a sum
of integrals on diagonals of length $p+n$ with products having a total order
of derivatives given by $\sum_{l=1}\sum_{j=2}^{n+1}\alpha
_{j}^{(l)}+\sum_{l=1}^{d}\sum_{i=1}^{p}\beta _{i}^{(l)}$. Hence we obtain a
sum of products, whose total order of derivatives is $\sum_{l=1}^{d}%
\sum_{j=2}^{n+1}\alpha _{j}^{(l)}+\sum_{l=1}^{d}\sum_{i=1}^{p}\beta
_{i}^{(l)}+\sum_{l=1}^{d}\alpha _{1}^{(l)}=\left\vert \alpha \right\vert
+\left\vert \beta \right\vert .$

For $k\geq 2$ we have (in connection with Lemma \ref{partialshuffle}) from
the induction hypothesis that 

\begin{align*}
\int_{\Delta _{\theta ,t}^{n+1}}f_{1}(s_{1})\dots f_{k}(s_{k})\int_{\Delta
_{\theta ,s_{k}}^{p}}g_{1}(r_{1})\dots g_{p}(r_{p})& dr_{p}\dots
dr_{1}f_{k+1}(s_{k+1})\dots f_{n+1}(s_{n+1})ds_{n+1}\dots ds_{1} \\
=\int_{\theta }^{t}f_{1}(s_{1})\int_{\Delta _{\theta
,s_{1}}^{n}}f_{2}(s_{2})\dots f_{k}(s_{k})& \int_{\Delta _{\theta
,s_{k}}^{p}}g_{1}(r_{1})\dots g_{p}(r_{p})dr_{p}\dots dr_{1} \\
& \times f_{k+1}(s_{k+1})\dots f_{n+1}(s_{n+1})ds_{n+1}\dots ds_{2}ds_{1} \\
=\sum_{\sigma \in A_{n,p}}\int_{\theta }^{t}f_{1}(s_{1})\int_{\Delta
_{\theta ,s_{1}}^{n+p}}& h_{1}^{\sigma }(w_{1})\dots h_{n+p}^{\sigma
}(w_{n+p})dw_{n+p}\dots dw_{1}ds_{1},
\end{align*}%

where each of the products $h_{1}^{\sigma }(w_{1})\cdot \dots \cdot
h_{n+p}^{\sigma }(w_{n+p})$ have a total order of derivatives given by $%
\sum_{l=1}\sum_{j=2}^{n+1}\alpha
_{j}^{(l)}+\sum_{l=1}^{d}\sum_{i=1}^{p}\beta _{i}^{(l)}.$ Thus we get a sum
with respect to a set of permutations $A_{n+1,p}$ with products having a
total order of derivatives which is%

\begin{equation*}
\sum_{l=1}^{d}\sum_{j=2}^{n+1}\alpha
_{j}^{(l)}+\sum_{l=1}^{d}\sum_{i=1}^{p}\beta _{i}^{(l)}+\sum_{l=1}^{d}\alpha
_{1}^{(l)}=\left\vert \alpha \right\vert +\left\vert \beta \right\vert .
\end{equation*}
\end{proof}

\end{document}